\documentclass{llncs}

\usepackage{graphicx}      

\usepackage{amsmath} 

\usepackage{amssymb} 

\begin{document}


\title{Doubly autoparallel structure \\
on the probability simplex}

\author{Atsumi Ohara\inst{1} \and Hideyuki Ishi\inst{2,3}}
\institute{University of Fukui, Fukui 910-8507, Japan \\ \email{ohara@fuee.u-fukui.ac.jp}
\and
Nagoya University, Furo-cho, Nagoya 464-8602, Japan \\ 
\email{hideyuki@math.nagoya-u.ac.jp}
\and
JST, PRESTO, 4-1-8, Honcho, Kawaguchi 332-0012, Japan
}

\maketitle              

\begin{abstract}
On the probability simplex, we can consider 
the standard information geometric structure with the 
e- and m-affine connections mutually dual with respect to the Fisher metric.
The geometry naturally 
defines submanifolds simultaneously autoparallel 
for the both affine connections, which we call 
{\em doubly autoparallel submanifolds}.

In this note we discuss their several interesting common properties.
Further, we algebraically characterize 
doubly autoparallel submanifolds on the probability simplex and 
give their classification.

\end{abstract}

\section{Introduction}
Let us consider information geometric structure \cite{AN} 
$(g,\nabla, \nabla^*)$ on a manifold ${\cal M}$, 
where $g,\nabla, \nabla^*$ are, respectively, 
a Riemannian metric and a pair of torsion-free affine connections satisfying 
\[
	Xg(Y,Z)=g(\nabla _X Y, Z)+g(Y, \nabla^*_X Z), 
	\quad \forall X, Y, Z \in {\cal X}({\cal M}).
\]
Here, ${\cal X}({\cal M})$ denotes a set of vector fields on ${\cal M}$.
Such a manifold with the structure $(g, \nabla, \nabla^*)$ is called 
a {\em statistical manifold} and we say $\nabla$ and $\nabla^*$ are 
{\em mutually dual} with respect to $g$.
When curvature tensors of $\nabla$ and $\nabla^*$ vanish, 
the statistical manifold is said {\em dually flat}.
For a statistical manifold, we can introduce a one-parameter family 
of affine connections called {\em $\alpha$-connection}:
\[
	\nabla^{(\alpha)}=\frac{1+\alpha}{2}\nabla 
		+ \frac{1-\alpha}{2} \nabla^*, \quad \alpha \in {\bf R}.
\]
It is seen that $\nabla^{(\alpha)}$ and $\nabla^{(-\alpha)}$ are mutually dual 
with respect to $g$.

In a statistical manifold, we can naturally define a submanifold ${\cal N}$ 
that is simultaneously autoparallel with respect to both $\nabla$ and $\nabla^*$.
\begin{definition}
Let $({\cal M}, g, \nabla, \nabla^*)$ be a statistical manifold and 
${\cal N}$ be its submanifold.
We call ${\cal N}$ {\em doubly autoparallel} in ${\cal M}$ 
when the followings hold:
\[
	\nabla_X Y \in {\cal X}({\cal N}), \; 
	\nabla^*_X Y \in {\cal X}({\cal N}),	
	\quad \forall X, Y \in {\cal X}({\cal N}).
\]
\end{definition}

We immediately see that 
doubly autoparallel submanifolds ${\cal N}$ possess 
the following properties: (Note that 4) and 5) hold if ${\cal M}$ is dually flat.) 

\begin{proposition}
The following statements are equivalent:
\begin{itemize}
\item[1)] a submanifold ${\cal N}$ is doubly autoparallel (DA),
\item[2)] a submanifold ${\cal N}$ is autoparallel 
w.r.t. to $\nabla^{(\alpha)}$ for two different $\alpha$'s,
\item[3)] a submanifold ${\cal N}$ is autoparallel 
w.r.t. to $\nabla^{(\alpha)}$ 
for all $\alpha$'s,
\item[4)] the $\alpha$-geodesics connecting two points on ${\cal N}$ 
(if they exist) lay in ${\cal N}$ for all $\alpha$'s,
\item[5)] a submanifold ${\cal N}$ is affinely constrained 
in both $\nabla$- and $\nabla^*$-affine coordinates of ${\cal M}$.
\end{itemize}
Furthermore, when ${\cal M}$ is dually flat and ${\cal N}$ is DA, 
the $\alpha$-projections 
to ${\cal N}$ (if they exist) are unique for all $\alpha$'s.
\end{proposition}

The concept of doubly autoparallelism has sometimes appeared 
but played important roles in several applications of 
information geometry \cite{Oh99,Oh04,OW10,UO04}.
However, the literature mostly treat information geometry 
of positive definite matrices or symmetric cones, 
and the study for statistical models has not been exploited yet. 

In this note, we consider doubly autoparallel structure 
on the probability simplex, which can be identified with 
probability distributions on discrete and finite sample spaces.
As a result, we give an algebraic characterization and classification of 
doubly autoparallel submanifolds in the probability simplex. 

Such manifolds commonly possess the above interesting properties.
Hence, the obtained results might be expected to give 
a useful insight into constructing statistical models 
for wide area of applications in information science, 
mathematics and statistical physics and so on \cite{ITA04,MA04,Montufar13}. 
Further, it should be mentioned that 
Nagaoka has recently reported the significance of this concept  
in study of Markov equivalence for statistical models \cite{Nag17}.

\section{Preliminaries}
\label{prel}
\subsection{Information geometry of ${\cal S}^n$ and ${\bf R}_+^{n+1}$}
Let us represent an element $p \in {\bf R}^{n+1}$ 
with its components $p_i, \; i=1,\cdots,n+1$ 
as $p=(p_i) \in {\bf R}^{n+1}$.
Denote, respectively, the positive orthant by 
\[
	{\bf R}^{n+1}_{+}:=\{p=(p_i) \in {\bf R}^{n+1}| p_i >0, \; 
	i=1, \cdots, n+1\},
\]
and the relative interior of the probability simplex by
\[
	{\cal S}^n:= \left\{ p \in {\bf R}^{n+1}_{+} \left| 
		\sum_{i=1}^{n+1} p_i =1 \right. \right\}.
\]
For a subset ${\cal Q} \subset {\bf R}_+^{n+1}$ and 
an element $p \in {\cal Q}$, we simply write
\[
	\log {\cal Q}:=\{\log p | p \in {\cal Q}\}, \quad 
	\log p:=(\log p_i) \in {\bf R}^{n+1}.
\]

Each element $p$ in the closure of ${\cal S}^n$ denoted by 
${\rm cl}{\cal S}^n$ can be identified with 
a discrete probability distribution for the sample space 
${\rm \Omega}=\{1,2, \cdots, n, n+1 \}$.
However, we only consider distributions $p(X)$ with positive probabilities, 
i.e., $p(i)=p_i>0, \; i=1,\cdots, n+1$, defined by
\[
	p(X)=\sum_{i=1}^{n+1} p_i \delta_i(X), \quad \delta_i(j)=\delta_i^j \;
	(\mbox{the Kronecker's delta}),
\]
which is identified with ${\cal S}^n$.

A statistical model in ${\cal S}^n$ is represented with 
parameters $\xi=(\xi_j), \; j=1,\cdots,d \le n$ by
\[
	p(X;\xi)=\sum_{i=1}^{n+1} p_i(\xi) \delta_i(X),
\]
where each $p_i$ is a function of $\xi$.
For example, $p_i=\xi_i, \; i=1,\cdots,n$ with the condition 
$\sum_{i=1}^n \xi_i <1$ is the full model, i.e., 
\[
	p(X;\xi)=\sum_{i=1}^n \xi_i \delta_i(X) 
		+\left( 1-\sum_{i=1}^n \xi_i \right) \delta_{n+1}(X)
\]
For the submodel, $\xi^j, \; j=1,\cdots,d<n$ can be also regarded as 
coordinates of the corresponding submanifold in ${\cal S}^n$. 

The standard information geometric structure on ${\cal S}^n$ \cite{AN} 
denoted by $(g,\nabla^{({\rm e})},\nabla^{({\rm m})})$ are composed of 
the pair of flat affine connections $\nabla^{({\rm e})}$ and $\nabla^{({\rm m})}$.
The affine connections $\nabla^{({\rm e})}=\nabla^{(1)}$ and 
$\nabla^{({\rm m})}=\nabla^{(-1)}$ are respectively called 
the {\em exponential connection} and the {\em mixture connection}. 
They are mutually dual with respect to the Fisher metric $g$.

By writing $\partial_i :=\partial/\partial \xi_i, \;i=1,\cdots,n $, 
they are explicitly represented  
as follows:
\begin{eqnarray}
	g_{ij}(p) & = & \sum_{X \in {\rm \Omega}} 	
		p(X)(\partial_i \log p(X))(\partial_j \log p(X)), 
		\quad i,j=1,\cdots,n, \nonumber \\
	\Gamma_{ij,k}^{({\rm m})}(p) 
	 &=&  \sum_{X \in {\rm \Omega}} 	
		p(X)(\partial_i \partial_j p(X))(\partial_k \log p(X)) 
		\quad \quad i,j,k=1,\cdots,n, 
\label{m-connex} \\
	\Gamma_{ij,k}^{({\rm e})}(p) 
	& = & \sum_{X \in {\rm \Omega}} 	
		p(X)(\partial_i \partial_j \log p(X))(\partial_k \log p(X)), 
		\quad i,j,k=1,\cdots,n. 
\label{e-connex} 
\end{eqnarray}

There exist two special coordinate systems.
The one is the {\em expectation coordinate} 
$\eta_i:=\sum_{X \in {\rm \Omega}}p(X)\delta_i(X)=p_i, \; i=1,\cdots,n$, 
which is 
$\nabla^{({\rm m})}$-affine from (\ref{m-connex}).
It implies that if each $\eta_i$ is an affine function
of all the model parameters $\xi_i$'s, 
then the statistical model is {\em $\nabla^{({\rm m})}$-autoparallel} 
(or sometimes called {\em m-flat}).

The other is the {\em canonical coordinate} $\theta^i$, 
which is defined by
\begin{equation}
	\theta^i:=\log \left(\frac{p_i}{1-\sum_{i=1}^n p_i}\right),
	\quad i=1,\cdots,n.
\label{e-affine}
\end{equation}
Since $\theta^i$'s satisfy
\[
	p(X)=\exp \left\{\sum_{i=1}^n 
		\theta^i \delta_i(X) -\psi(\theta) \right\}, \quad
	\psi(\theta):=\log \left(1+\sum_{i=1}^n \exp \theta^i \right),
\]
they are $\nabla^{({\rm e})}$-affine from (\ref{e-connex}).
Hence, it implies that if each $\theta^i$ is an affine function 
of all the model parameters $\xi_i$'s, then the statistical model is 
{\em $\nabla^{({\rm e})}$-autoparallel} 
(or sometimes called {\em e-flat}).

Note that from the property of the expectation coordinates, 
a $\nabla^{({\rm e})}$-autoparallel submanifold in ${\cal S}^n$, 
denoted by $M$, should be 
represented by $M=W \cap {\cal S}^{n}$ 
for a certain subspace $W \subset {\bf R}_+^{n+1}$.
This fact is used later.

Finally, we introduce information geometric structure 
$(\tilde g,\tilde \nabla^{({\rm e})},\tilde \nabla^{({\rm m})})$ on ${\bf R}_+^{n+1}$.
The structure $(g,\nabla^{({\rm e})},\nabla^{({\rm m})})$ 
on ${\cal S}^n$ is a submanifold geometry induced from  
this ambient structure.
For arbitrary coordinates $\tilde \xi_i, \; i=1,\cdots,n+1$ 
of ${\bf R}_+^{n+1}$, let us take 
$\tilde \partial_i:=\partial/\partial \tilde \xi_i$.
Then their components are given by
\begin{eqnarray}
	\tilde g_{ij}(p) &=&  \sum_{X \in {\rm \Omega}} 	
		p(X)(\tilde \partial_i \log p(X))(\tilde \partial_j \log p(X)),
		\quad i,j=1,\cdots,n+1, \nonumber \\
	\tilde \Gamma_{ij,k}^{({\rm m})}(p) 
	& = & \sum_{X \in {\rm \Omega}} 	
		p(X)(\tilde \partial_i \tilde \partial_j p(X))
		(\tilde \partial_k \log p(X)),
		\quad i,j,k=1,\cdots,n+1, 
\label{m_aff_amb} \\
	\tilde \Gamma_{ij,k}^{({\rm e})}(p)
	& = & \sum_{X \in {\rm \Omega}} 	
		p(X)(\tilde \partial_i \tilde \partial_j \log p(X))
		(\tilde \partial_k \log p(X)),
		\quad i,j,k=1,\cdots,n+1.
\label{e_aff_amb}
\end{eqnarray}
Thus, we find
that $p_i$'s are $\tilde \nabla^{({\rm m})}$-affine coordinates 
and $\log p_i$'s are $\tilde \nabla^{({\rm e})}$-affine coordinates, respectively, 
from (\ref{m_aff_amb}), (\ref{e_aff_amb}) and 
$\log p(X)=\sum_{X \in \Omega} (\log p_i)\delta_i(X) $.


\subsection{An example}
\label{expr}
{\bf Example:} 
Let $v^{(k)}=(\delta_i^k) \in {\bf R}^{n+1}, \; k=1,\cdots,n+1$ represent 
vertices on ${\rm cl}{\cal S}^n$. 
Take a vector $v^{(0)} \in {\bf R}^{n+1}_{+}$ that is linearly independent 
of $\{v^{(k)} \}_{k=1}^d \; (d <n)$ and define a subspace of 
dimension $d+1$ by $W={\rm span}\{v^{(0)},v^{(1)}, \cdots, v^{(d)}\}$.
Then $M={\cal S}^n \cap W$ is doubly autoparallel.

We show this for the case $d=2$ but 
similar arguments hold for general $d$.
For the simplicity we take the following $v^{(i)}, \; i=0,1,2$ :
\begin{eqnarray*}
& \displaystyle	
	v^{(0)}=(0 \; 0 \; p_3 \; \cdots \; p_{n+1})^T, \quad 
	\sum_{i=3}^{n+1} p_i=1, \; p_i >0, \; i=3,\cdots,n+1, & \\
&	v^{(1)}=(1 \; 0 \; \cdots \; 0)^T, \quad 
	v^{(2)}=(0 \; 1 \; 0\; \cdots \; 0)^T,
	\quad \mbox{($\; \cdot \; ^T$ denotes the transpose)}. &
\end{eqnarray*}
Since for $p \in M$ we have a convex combination by parameters $\xi_i$ as
\[
	p=\xi_1 v^{(1)}+\xi_2 v^{(2)}+(1-\xi_1-\xi_2)v^{(0)},
\]
the expectation coordinates $\eta_i$'s are 
\begin{eqnarray*}
&	\eta_1=\xi_1, \; \eta_2=\xi_2, \; 
		\eta_i=(1-\xi_1-\xi_2)p_i, \; i=3,\cdots,n+1, & \\
&	(\xi_1>0, \; \xi_2 >0, \; \xi_1+\xi_2 <1).&
\end{eqnarray*}
Thus, each $\eta_i$ is affine in $\xi_i, \; i=1,2$.

On the other hand, the canonical coordinates $\theta^i$'s are 
\begin{eqnarray*}
&	\theta^1=\zeta_1, \; \theta^2=\zeta_2, \; 
	\theta^i=\log p_i +c, \; i=3,\cdots,n+1, & \\
&	(\zeta_i=\log \{\xi_i/(1-\xi_1-\xi_2)\}, \; i=1,2, \; c=-\log p_{n+1}). &
\end{eqnarray*}
Thus, each $\theta^i$ is affine in parameters $\zeta_i, \; i=1,2$.
Hence, $M$ is doubly autoparallel.



\subsection{Denormalization}
\begin{definition}
Let $M$ be a submanifold in ${\cal S}^n$.
The submanifold $\tilde M$ in ${\bf R}_+^{n+1}$ defined by
\[
	\tilde M=\{\tau p \in {\bf R}_+^{n+1}| \; p \in M, \; \tau >0 \}
\]
is called a {\em denormalization} of $M$ \cite{AN}.
\end{definition}
\begin{lemma}
\label{lem1}
A submanifold $M$ is $\nabla^{(\pm 1)}$-autoparallel in ${\cal S}^n$ 
if and only if the denormalization $\tilde M$ is 
$\tilde \nabla^{(\pm 1)}$-autoparallel in ${\bf R}_{+}^{n+1}$.
\end{lemma}

A key observation derived from the above lemma is as follows:

Since $p_i, \; i=1,\cdots,n+1$ are $\nabla^{({\rm m})}$-affine coordinates 
for ${\cal S}^n$, a submanifold $M \subset {\cal S}^n$ is 
$\nabla^{({\rm m})}$-autoparallel if and only if it is represented as
$M=W \cap {\cal S}^n$ for a subspace $W \subset {\bf R}^{n+1}$.
Hence, by definition $\tilde M$ is nothing but
\begin{equation}
	\tilde M=W \cap {\bf R}_+^{n+1}.
\label{m_rep}
\end{equation}

On the other hand, since the coordinates 
$\log p_i, \; i=1,\cdots, n+1$ for ${\bf R}^{n+1}$ are 
$\tilde \nabla^{({\rm e})}$-affine,
$\tilde M$ is $\tilde \nabla^{({\rm e})}$-autoparallel if and only if 
there exist a subspace $V \subset {\bf R}^{n+1}$ and a constant element 
$b \in {\bf R}^{n+1}$ satisfying 
\begin{equation}
	\log \tilde M =b+V,
\label{e_rep}
\end{equation}
where ${\rm dim}W={\rm dim}V$.
If so, $M$ is also $\nabla^{({\rm e})}$-autoparallel from lemma \ref{lem1}.

Thus, we study conditions for the denormalization $\tilde M$ to have 
simultaneously dualistic representations (\ref{m_rep}) and (\ref{e_rep}), 
which is equivalent to doubly autoparallelism of $M$.

\section{Main results}
First we introduce an algebra $({\bf R}^{n+1}, \circ)$ via 
the Hadamard product $\circ$, i.e.,
\[
	x \circ y=(x_i) \circ (y_i):=(x_iy_i), \quad x,y \in {\bf R}^{n+1},
\]
where the identity element $e$ and an inverse $x^{-1}$ are 
\[
	e={\bf 1}, \quad x^{-1}=\left( \frac{1}{x_i} \right),
\]
respectively. 
Here, ${\bf 1} \in {\bf R}^{n+1}_+$ is 
the element all the components of which are one. 
Note that the set of invertible elements 
\[
	{\cal I}:=\{x=(x_i) \in {\bf R}^{n+1}| x_i \not=0,\; i=1,\cdots,n+1 \}
\]
contains ${\bf R}^{n+1}_+$.
We simply write $x^k$ for the powers recursively defined by 
$x^k=x \circ x^{k-1}$.

For an arbitrarily fixed $a \in {\cal I}$ the algebra 
$({\bf R}^{n+1}, \circ)$ induces another algebra called a {\em mutation} 
$({\bf R}^{n+1}, \circ_{a^{-1}})$, the product of which is defined by
\[
	x \circ_{a^{-1}} y:=x \circ a^{-1} \circ y=(x_iy_i/a_i), 
	\quad x,y \in {\bf R}^{n+1},
\]
with its identity element $a$.
We write $x^{(\circ a^{-1})k}$ for the powers by $\circ_{a^{-1}}$.

We give a basic result in terms of $({\bf R}^{n+1}, \circ)$.
\begin{theorem}
Assume that $a \in \tilde M=W \cap {\bf R}_{+}^{n+1}$.
Then, there exists a subspace $V$ satisfying
\[
	\log \tilde M=\log \left\{(a+W) \cap {\bf R}_{+}^{n+1} \right\} 
	= \log a +V
\]
if and only if the following two conditions hold:
\[
	1)\; V=a^{-1}\circ W, \quad 
	2)\; \forall u, w \in W, \;\; u \circ a^{-1} \circ w \in W.
\] 
\end{theorem}
\begin{proof}
(``only if" part): For all $w \in W$ and small $t \in {\bf R}_+$, 
we have $\log(a+tw) \in \log a +V$.
Hence, it holds that
\[
	\left. \frac{d}{dt} \log(a+tw) \right|_{t=0} 
	= a^{-1} \circ w \in V.
\]
Thus, the condition 1) holds.

Similarly, for all $u, w \in W$ and 
small $t \in {\bf R}_+$ and $s \in {\bf R}_+$,
we have $\log(a+su+tw) \in \log a +V$ and obtain
\[
	\left. \frac{\partial}{\partial s} 
	\left( \left. \frac{\partial}{\partial t} \log(a+su+tw) \right|_{t=0} 
	\right) \right|_{s=0}=-a^{-1} \circ u \circ a^{-1} \circ w \in V. 
\]
Hence, we see that the condition 2) holds, using the condition 1). 

(``if" part): For $w=(w_i) \in W$ 
satisfying $a+w \in {\bf R}^{n+1}_+$, 
take $t \in {\bf R}_+$ be larger than $(1+\max_i\{ w_i/a_i \})/2$.
Then there exists $u=(u_i) \in W$ satisfying
\begin{equation}
	a+w=t a+u, \quad ta_i > |u_i|, \; i=1, \cdots,n+1.
\label{ineq_expand}
\end{equation}
Hence, we have
\begin{eqnarray}
	\log(a+w)&=&\log(ta+u)=\log \{(ta)\circ(e+(ta)^{-1}\circ u)\} 
	\nonumber \\
	&=&(\log t)e + \log a + \log \{e+(ta)^{-1}\circ u\}. 
\label{expand}
\end{eqnarray}
Using the inequalities in (\ref{ineq_expand}) and the Taylor series
\[
	\log(1+x)=\sum_{k=1}^\infty \frac{(-1)^{k-1}}{k}x^k, \quad |x| <1,
\]
we expand the right-hand side of (\ref{expand}) as
\[
	(\log t)e + \log a + a^{-1} \circ 
		\left( \frac{1}{t}\sum_{k=1}^\infty 
		\frac{1}{k} \left( \frac{-1}{t} \right)^{k-1} 
		u^{(\circ a^{-1})k}
		\right).
\]
Since each $u^{(\circ a^{-1})k}$ belongs to $W$ from the condition 2), 
the third term is in $V$ by the condition 1).
Further the condition 1) implies that $e \in V$, so is the first term.
This completes the proof.
\end{proof}

\begin{remark}
i) The condition 2) claims that $W$ is a {\em subalgebra} of 
$({\bf R}^{n+1},\circ_{a^{-1}})$.

\noindent 
ii) The affine subspace $\log a +V$ is independent of the choice of 
$a \in \tilde M=W \cap {\bf R}_{+}^{n+1}$.
This follows from the proof of ``if" part by taking $a'=a+w$.
\end{remark}

The following algebraic characterization of 
doubly autoparallel submanifold in ${\cal S}^n$ is immediate from
the above theorem and lemma \ref{lem1} in section \ref{prel}.
\begin{corollary}
A $\nabla^{({\rm m})}$-autoparallel submanifold 
$M=W\cap {\cal S}^n$ is doubly autoparallel if and only if 
the subspace $W$ is a subalgebra of $({\bf R}^{n+1},\circ_{a^{-1}})$ 
with $a \in \tilde M$.
\end{corollary}

Finally, in order to answer a natural question what structure is 
necessary and sufficient for $W$, we classify  
subalgebras of $({\bf R}^{n+1},\circ_{a^{-1}})$.
Let $q$ and $r$ be integers 
that meet $q \ge 0$, $r>0$ and $q+r={\rm dim}W <n+1$. 
Define integers $n_l, \;l=1,\cdots,r$ satisfying
\[
	 q+\sum_{l=1}^r n_l =n+1, \quad 2 \le n_1 \le \cdots \le n_r.
\] 
Constructing subvectors $a_l \in {\bf R}^{n_l}_{+},\; l=1,\cdots,r$ 
with components arbitrarily extracted from $a \in W \cap {\bf R}_{+}^{n+1}$ 
without duplications, 
we denote by $\Pi$ the permutation matrix that meets
\begin{equation}
	(a_0^T \; a_1^T \; \cdots \; a_r^T )^T=\Pi a, 
\label{permutation}
\end{equation}
where the subvector $a_0 \in {\bf R}^q$ is composed of 
the remaining components in $a$.
We give the classification via the canonical form 
for $W_0=\Pi W=\{w' \in{\bf R}^{n+1}| w'=\Pi w, \; w \in W \}$ 
based on this partition instead of the original form for $W$.
\begin{theorem}
For the above setup, $W$ is a subalgebra of $({\bf R}^{n+1},\circ_{a^{-1}})$ 
with $a \in \tilde M$ if and only if $W$ is isomorphic to
$
	{\bf R}^q \times {\bf R}a_1 \times \cdots \times {\bf R}a_r
$
and represented by $\Pi ^{-1} W_0$, where 
\[
	W_0=\{(y^T \; t_1 a_1^T \; \cdots \; t_r a_r^T )^T \in {\bf R}^{n+1}
	 | \;
	 \forall y \in {\bf R}^q, \; a_l \in {\bf R}^{n_l}_{+},\; \forall 
	t_l \in {\bf R}, \; l=1,\cdots,r \}.
\]
%

\end{theorem}
\begin{proof}
(``only if" part) 
Let $V$ be a subspace in ${\bf R}^{n+1}$ defined by $V=a^{-1} \circ W$.
Then it is straightforward that $W$ is a subalgebra of 
$({\bf R}^{n+1},\circ_{a^{-1}})$ if and only if $V$ is a subalgebra 
of $({\bf R}^{n+1},\circ)$.
Using this equivalence, we consider the necessity condition.

Since $e$ and $x^k$ are in $V$ for any $x=(x_i) \in V$ 
and positive integer $k$, the square matrix $\Xi$ defined by
\[
	\Xi:=\left(e \; x \; \cdots \; x^n \right)=
	\left( \begin{array}{cccc}
	1 & x_1 &  \cdots & x_1^n  \cr
	1 & x_2 & \cdots & x_2^n  \cr
	\vdots & \vdots & \ddots & \vdots  \cr
	1 & x_{n+1} & \cdots & x_{n+1}^n \cr
	\end{array} \right)
\]
is singular.
The determinant of the Vandermonde's matrix $\Xi$ is 
calculated using the well-known formula, as
\begin{eqnarray*}
\det \Xi=(-1)^{(n+1)n/2} \left(\prod_{i<j}  (x_i-x_j) \right).
\end{eqnarray*}
Hence, it is necessary for $x$ to belong to $V$ that 
\begin{equation}
	\exists (i,j), \; x_i=x_j.
\label{nec1}
\end{equation}

Denoting basis vectors of $V$ by 
$v^{(k)}=(v^{(k)}_i) \in {\bf R}^{n+1}, 
\; k=1,\cdots,q+r(={\rm dim}V)$, 
we can represent any $x$ as $x=\sum_{k=1}^{q+r} \alpha_k v^{(k)}$ 
using a coefficient vector $(\alpha_k) \in {\bf R}^{q+r}$.
Hence, the necessary condition (\ref{nec1}) is equivalent to 
\begin{equation}
	\forall (\alpha_k) \in {\bf R}^{q+r}, \; \exists (i,j), \quad
	\sum_{k=1}^{q+r} \alpha_k (v_i^{(k)} -v_j^{(k)})=0.
\label{nec2}
\end{equation}
It is easy to see, by contradiction, that 
(\ref{nec2}) implies the following condition: 
\begin{equation}
	\exists (i,j), \; \forall k, \quad 
	v_i^{(k)} = v_j^{(k)}.
\label{nec3}
\end{equation}
By normalization $v_i^{(k)} = v_j^{(k)}=1$ for $(i,j)$ 
satisfying (\ref{nec3}) and 
a proper permutation of $i=1,\cdots,n+1$, 
we find that possible canonical form
of subspace $V$, which we denote by $V_0$, is restricted to
\[
	V_0=\{(z^T \; t_1 {\bf 1}^T \; \cdots \; t_r {\bf 1}^T )^T 
	\in {\bf R}^{n+1} | \;
	\forall z \in {\bf R}^q, \; t_l{\bf 1} \in {\bf R}^{n_l},\; 
	\forall t_l \in {\bf R}, \; l=1,\cdots,r \}
\] 
for $q,r$ and $n_l, \; l=1,\cdots,r$ given in the setup.
Using the above permutation as $\Pi$ in the setup, i.e., $V=(\Pi^{-1}V_0)$, 
we have an isomorphic relation 
$W=a \circ (\Pi^{-1}V_0)$.
Thus, this means that $W_0=\Pi W=(\Pi a)\circ V_0$.

(``if" part) Conversely it is easy to confirm $V_0$ is 
a subalgebra of $({\bf R}^{n+1},\circ)$. 
We show that any other proper subspaces in $V_0$ cannot be 
a subalgebra with $e$, 
except for the trivial cases where several $t_l$'s or components of $z=(z_i)$ 
are fixed to be zeros\footnote{These cases contradict the fact that 
$e \in V_0$.} 
or equal to each other\footnote{These cases correspond to 
choosing smaller $q$ or $r$ in the setup.}.

Consider a subspace $V' \subset V_0$ with 
nontrivial linear constraints between $t_l$'s and $z_i$'s.
If $V'$ is a subalgebra, then for all $x \in V'$ and integer $m$ we have
\[
	V' \ni x^m=\left( (z^m)^T \;\; t_1^m {\bf 1}^T \; \cdots 				\; t_r^m {\bf 1}^T\right)^T, \qquad z^m=(z_i^m) \in {\bf R}^q,
\]
where $t_l^m$'s and $z_i^m$'s should satisfy the same linear constraints.
We, however, find this is impossible by the similar arguments 
with the Vandermonde's matrix in the ``only if" part.
This completes the proof.


\end{proof}

\noindent
{\bf Example (continued from section \ref{expr}):} 
As $a \in \tilde M=W \cap {\bf R}_+^{n+1}$ we set
\[
	a=(1 \;\; 2 \;\; p_3 \; \cdots \; p_{n+1} )^T, \; a_0=(1 \;\; 2)^T, \; 
	a_1=(p_3 \; \cdots \; p_{n+1})^T.
\]
Then we have $q=2, \; r=1, \; n_1=n-1$ and need no permutation, i.e., $W=W_0$.
Since every element in $W$ can be represented by
\[
	w=(\xi_1 \;\; \xi_2 \;\; t p_3 \; \cdots \; t p_{n+1} )^T, \quad 
	\xi_1,\; \xi_2,\; t \in {\bf R}
\]
we can confirm $W$ is a subalgebra of $({\bf R}^{n+1},\circ_{a^{-1}})$ 
and 
\[
	V_0=V=a^{-1} \circ W=\{(z^T \;\; t {\bf 1}^T )^T
	\in {\bf R}^{n+1} | \;
	 \forall z \in {\bf R}^2, \; t {\bf 1} \in {\bf R}^{n-1}, \; 
	\forall t \in {\bf R} \}.
\]

\section{Concluding remarks}
We have studied doubly autoparallel structure of statistical models 
in the family of probability distributions on discrete and finite sample space.
Identifying it by the probability simplex and 
using the mutation of Hadamard product, we give an algebraic characterization 
of doubly autoparallel submanifolds and their classification.

\section*{Acknowledgements}
A. O. was partially supported by JSPS Grant-in-Aid (C) 15K04997.


\begin{thebibliography}{5}
%
\bibitem{AN}
S-I. Amari and H. Nagaoka,
Methods of Information Geometry,
Trans. Math. Monogr. vol 191
(American Mathematical Society and Oxford University Press) (2000).
\bibitem{ITA04}
S. Ikeda, T. Tanaka and S-I. Amari,
Stochastic reasoning, free energy, and information geometry,
Neural Computation, 16, 1779-1810 (2004).
\bibitem{MA04}
F. Matus and N. Ay,
On maximization of the information divergence form an exponential family,
Proc. WUPES'03, 199-204 (2004).
\bibitem{Montufar13}
G. F. Mont\'ufar,
Mixture decompositions of exponential families using a decomposition of 
their sample spaces,
Kybernetika, 49, 1, 23-39 (2013).
\bibitem {Nag17}
H. Nagaoka,
Information-geometrical characterization of statistical models which are 
statistically equivalent to probability simplex,
arXiv:1701.07736v2 (2017).
\bibitem{Oh99}
A. Ohara,
Information Geometric Analysis of an Interior Point Method for 
Semidefinite Programming,
Proc. of Geometry in Present Day Science 
(O. Barndorf-Nielsen and V. Jensen eds.) World Scientific, 49-74 (1999).
\bibitem{Oh04}
A. Ohara,
Geodesics for Dual Connections and Means on Symmetric Cones,
Integral Equations and Operator Theory, Vol.50, 537-548 (2004).
\bibitem{OW10}
A. Ohara and T. Wada,
Information Geometry of q-Gaussian Densities and Behaviours of Solutions to 
Related Diffusion Equations,
Journal of Physics A: Mathematical and Theoretical, 
Vol.43, 035002 (18pp.) (2010). 
\bibitem{UO04}
K. Uohashi and A. Ohara,
Jordan Algebras and Dual Affine Connections on Symmetric Cones,
Positivity, Vol. 8, No. 4, 369-378 (2004).
\end{thebibliography}
\end{document}